\documentclass[10pt]{amsart}
\usepackage{egastyle}

\usepackage{amssymb,amsmath,amscd,amsthm,amsfonts,wasysym,mathrsfs,amsxtra,stmaryrd}

\usepackage{graphicx,array,url}
\usepackage[vcentermath]{genyoungtabtikz}
\usepackage{tikz}
\usepackage[section]{placeins}
\usepackage{afterpage}
\usepackage{verbatim}
\input xy
\xyoption{all}

\newcommand{\OO}{\mathcal{O}}

\newcommand{\EE}{\mathcal{E}}

\newcommand{\image}{\textnormal{im}\,}
\newcommand{\kernel}{\textnormal{ker}\,}
\newcommand{\cokernel}{\textnormal{coker}\,}

\newcommand{\HHom}{\mathscr{H}om} 
\newcommand{\Hom}{\textnormal{Hom}}
\newcommand{\dimension}{\textnormal{dim}\,}
\newcommand{\codimension}{\textnormal{codim}\,}

\newcommand{\EExt}{\mathcal{E}xt}

\newcommand{\Eb}{{\ol{E}}}
\newcommand{\Ac}{\mathcal{A}}

\newcommand{\Fc}{\mathcal{F}}
\newcommand{\Gc}{\mathcal{G}}
\newcommand{\Lc}{\mathcal{L}}

\newcommand{\Qc}{\mathcal{Q}}
\newcommand{\Sc}{\mathcal{S}}

\newcommand{\Coh}{\mathrm{Coh}}
\newcommand{\Ap}{\mathcal{A}^p}

\newcommand{\arinj}{\ar@{^{(}->}}
\newcommand{\arsurj}{\ar@{->>}}
\newcommand{\areq}{\ar@{=}}

\newcommand{\ch}{\mathrm{ch}}

\newcommand{\wt}{\widetilde}
\newcommand{\ol}{\overline}

\newcommand{\Mor}{\mathcal{M}or}

\newcommand{\Quot}{{\mathrm{Quot}}}

\makeatletter
\newtheorem*{rep@theorem}{\rep@title}
\newcommand{\newreptheorem}[2]{%
\newenvironment{rep#1}[1]{%
 \def\rep@title{#2 \ref{##1}}%
 \begin{rep@theorem}}%
 {\end{rep@theorem}}}
\makeatother

\newreptheorem{theorem}{Theorem}

\usepackage{scalerel,stackengine}
\stackMath
\newcommand\reallywidehat[1]{%
\savestack{\tmpbox}{\stretchto{%
  \scaleto{%
    \scalerel*[\widthof{\ensuremath{#1}}]{\kern-.6pt\bigwedge\kern-.6pt}%
    {\rule[-\textheight/2]{1ex}{\textheight}}
  }{\textheight}%
}{0.5ex}}%
\stackon[1pt]{#1}{\tmpbox}%
}
\begin{document}

\title[a Relation between higher-rank PT stable objects and quotients of coherent sheaves]{A relation between higher-rank PT stable objects and  quotients of coherent sheaves}

\author[Jason Lo]{Jason Lo}
\address{Department of Mathematics \\
California State University, Northridge\\
18111 Nordhoff Street\\
Northridge CA 91330 \\
USA}
\email{jason.lo@csun.edu}


\begin{abstract}
On a smooth projective threefold, we construct an essentially surjective functor $\mathcal{F}$ from a category of two-term complexes to a category of quotients of coherent sheaves, and describe the fibers of this functor.  Under a coprime assumption on rank and degree, the domain of $\mathcal{F}$ coincides with the category of higher-rank PT stable objects, which appear on one side of Toda's higher-rank DT/PT correspondence formula.  The codomain of $\mathcal{F}$ is the category of objects that appear on one side of another correspondence formula by Gholampour-Kool, between the generating series of topological Euler characteristics of two types of quot schemes.
\end{abstract}

\maketitle
\tableofcontents

\section{Introduction}

On a smooth projective threefold $X$, Gholampour-Kool computed the generating series of some moduli spaces of slope stable sheaves of homological dimension at most one \cite{gholampour2017higher}.  An integral part of their argument was the following counting formula, where $\Quot_X(-,n)$ denotes the quot scheme of length-$n$ quotients of a coherent sheaf, $e(-)$ denotes the topological Euler characteristic, and $M(q)=\prod_{n=1}^\infty \frac{1}{(1-q^n)^n}$ is the MacMahon fuction:

\begin{thm}\cite[Theorem 1.1]{gholampour2017higher}\label{thm:GK1-1}
For any rank $r$ torsion-free sheaf $F$ of homological dimension at most 1 on a smooth projective threefold $X$, we have
\begin{equation}\label{eq:GKmainthm}
  \sum_{n=0}^\infty e(\Quot_X (F,n))q^n = M(q)^{re(X)} \sum_{n=0}^\infty e(\Quot_X (\EExt^1 (F,\OO_X),n))q^n.
\end{equation}
\end{thm}

On the other hand, on a smooth projective Calabi-Yau threefold $X$, Toda proved a correspondence formula between higher-rank Donaldson-Thomas (DT) and Pandharipande-Thomas (PT) invariants.  While DT invariants virtually count slope stable sheaves on $X$, PT invariants count PT stable objects in the derived category of coherent sheaves $D^b(X) = D^b(\Coh (X))$ on $X$.  PT stability is a type of polynomial stability on $D^b(X)$ in the sense of Bayer \cite{BayerPBSC}; a rank-one PT stable object of trivial determinant is exactly a stable pair
\begin{equation}\label{eq:PTstablepair}
\OO_X \overset{s}{\to} F
\end{equation}
 in the sense of Pandharipande-Thomas \cite{PT}, which we call a PT stable pair, where $F$ is a pure 1-dimensional sheaf and the cokernel of the morphism of sheaves $s$ is 0-dimensional.  The  properties of PT stable objects were studied and their moduli spaces constructed by the author in \cite{Lo1, Lo2, Lo4}.   For any  ample divisor $\omega$ on $X$ and any $(r,D,-\beta,-n) \in H^0(X) \oplus H^2(X) \oplus H^4(X) \oplus H^6(X)$ where $r \geq 1$ with $r,D\omega^2$ coprime, let us write  $\mathrm{DT} (r,D,-\beta, -n)$ to denote the DT invariant virtually counting $\mu_\omega$-stable sheaves of Chern character $(r,D,-\beta, -n)$, and $\mathrm{PT}(r,D,-\beta,-n)$ to denote  the PT invariant virtually counting PT stable objects of that Chern character.  Then  Toda's correspondence formula reads:

\begin{thm}\cite[Theorem 1.2]{Toda2}\label{thm:Toda1-2}
 For a fixed $(r,D,\beta)$, we have
\begin{equation}\label{eq:Todamainthm}
  \sum_{6n \in \mathbb{Z}} \mathrm{DT} (r,D,-\beta, -n)q^n = M( (-1)^r q)^{re(X)} \sum_{6n \in \mathbb{Z}} \mathrm{PT}(r,D,-\beta,-n)q^n.
\end{equation}
\end{thm}
The case $(r,D)=(1,0)$ of the formula \eqref{eq:Todamainthm}, i.e.\ rank-one DT/PT correspondence, was first conjectured in \cite{PT} and first proved by Bridgeland \cite{bridgeland2011hall}.  Toda also gave a proof under an assumption on the local structure of the moduli stacks involved \cite{toda2010curve}; the assumption was later removed in \cite{Toda2}.

In this article, we describe a relation between the  objects  that appears on the right-hand side of Gholampour-Kool's formula \eqref{eq:GKmainthm} and the  objects that appears on the right-hand side of Toda's formula \eqref{eq:Todamainthm}.  More precisely, on a smooth projective threefold $X$, we define a category $\EE_0$ of 2-term complexes in $D^b(X)$ with cohomology at degrees $-1, 0$, that contains all the PT semistable objects in $D^b(X)$.  The category $\EE_0$ also contains all the `frozen triples' in the sense of Sheshmani \cite{sheshmani2016higher}, which gives an alternative approach for  generalising Pandharipande-Thomas' stable pairs \eqref{eq:PTstablepair} to higher ranks.  We write $\Mor (\Coh (X))$ to denote the category where the objects are  morphisms of coherent sheaves on $X$, and morphisms are given by commutative squares in $\Coh (X)$.  For any coherent sheaf $A$ on $X$, let us  write $\Sc (A)$ to denote the full subcategory of $\Mor (\Coh (X))$ consisting of objects of the form $A \overset{q}{\to} Q$ where $Q$ is a 0-dimensional sheaf, and $q$ is a surjection of sheaves.  We construct a (contravariant) functor
\begin{equation*}
  \Fc : \EE_0 \to \coprod_{F \in \Coh (X), \mathrm{hd}(F)\leq 1} \Sc (\EExt^1 (F,\OO_X))^{op}
\end{equation*}
(see Definition \ref{para:mainthmprep1}) and prove our  main result:

\begin{thm}[Theorem \ref{thm:main1}]
The functor $\Fc$ is essentially surjective.  If we fix an ample class on $X$,  fix $r \in H^0(X), D \in H^2(X)$ such that $r,\omega^2D$ are coprime, restrict the domain of $\Fc$ to PT stable objects $E$ with $\ch_0(E)=-r, \ch_1(E)=-D$ and restrict the codomain by requiring $F$ above to be $\mu_\omega$-stable with $\ch_0(F)=r, \ch_1(F)=D$, then the restriction of $\Fc$ is also essentially surjective.
\end{thm}

In Section \ref{sec:Fcfiber}, we analyse the fibers of the functor $\Fc$.  We describe how to enumerate all the objects in a given fiber of $\Fc$ in Lemma \ref{lem:AG44-136-1}.  For two objects $E, \ol{E}$ of $\EE_0$, we  pin down the difference between  $E, \ol{E}$ being isomorphic in $D^b(X)$ and $\Fc (E), \Fc (\ol{E})$ being isomorphic in $\Mor (\Coh (X))$.    Finally, we recall a construction mentioned in Gholampour-Kool's work in which  a sheaf quotient $\EExt^1 (I_C,\OO_X) \twoheadrightarrow Q$, where $I_C$ is the ideal sheaf of a Cohen-Macaulay curve $C$ on $X$ and $Q$ is a 0-dimensional sheaf, can be used to construct a PT stable pair (i.e.\ a rank-one PT stable object).  We generalise this construction to higher ranks in \ref{para:GKrank1constr-higherrk}, so that given a higher-rank sheaf quotient, we produce   a higher-rank PT stable object.  We end the article with Lemma \ref{lem:inversecompGK}, which compares this higher-rank construction and  the functor $\Fc$ constructed in Section \ref{sec:esssurj}.



\subsection{Acknowledgements} The author would like to thank Yunfeng Jiang, Martijn Kool and Zhenbo Qin for answering his various questions, and Jun Li and Ziyu Zhang for helpful discussions.

\section{Preliminaries}

\subsection{Notation}

Unless otherwise stated, we will write $X$ for a smooth projective threefold in this article, $\Coh (X)$ for the category of coherent sheaves on $X$, and  $D^b(X) = D^b(\Coh (X))$ for the bounded derived category of coherent sheaves on $X$.

\paragraph For any category $\mathcal{C}$, we will write $\Mor (\mathcal{C})$ to denote the category of morphisms in $\mathcal{C}$.  That is, the objects of $\Mor (\mathcal{C})$ are morphisms $f : A \to B$ in $\mathcal{C}$, and a morphism between two objects $f : A \to B, f' : A' \to B'$ of $\Mor (\mathcal{C})$ is a commutative diagram in $\mathcal{C}$
\[
\xymatrix{
  A \ar[r]^f \ar[d] & B \ar[d] \\
  A' \ar[r]^{f'} & B'
}.
\]

\paragraph For any object $E \in D^b(X)$ and any subcategory $\mathcal{C}$ of $D^b(X)$, we will write $\Hom (\mathcal{C},E)=0$ to mean $\Hom_{D^b(X)} (C,E)=0$ for all $C \in \mathcal{C}$, and similarly for $\Hom (E,\mathcal{C})=0$.

\paragraph For any integer $d$, we will write $\Coh^{\leq d}(X)$ to denote the Serre subcategory of $\Coh (X)$ consisting of  sheaves $E$  supported in dimension at most $d$.  We will also write $\Coh^{\geq d}(X)$ to denote the full subcategory of $\Coh (X)$ consisting of sheaves $E$ such that $\Hom (\Coh^{\leq d-1}(X),E)=0$, i.e.\ sheaves $E$ that have no subsheaves supported in dimension $d-1$ or lower.  Then we set $\Coh^{=d}(X) = \Coh^{\leq d}(X) \cap \Coh^{\geq d}(X)$, which is the category of pure $d$-dimensional sheaves on $X$.

\paragraph Given any object $E \in D^b(X)$ and any integer $i$, we will write $H^i(E)$ to denote the degree-$i$ cohomology of $E$ with respect to the standard t-structure on $D^b(X)$.  We then define
\[
  D^{\geq i}_{\Coh (X)} = \{ E \in D^b(X) : H^k(E)=0 \text{ for all }k < i \}
\]
and similarly  $D^{\leq i}_{\Coh (X)}$.  For any integers $i \leq j$, we  set
\[
  D^{[i,j]}_{\Coh (X)} = D^{\geq i}_{\Coh (X)} \cap D^{\leq j}_{\Coh (X)}.
\]

\paragraph Given a coherent sheaf $F$ on $X$, we will refer to the dimension (resp.\ codimension) of $\mathrm{supp}(F)$ simply as the dimension (resp.\ codimension) of $F$, and denote it as $\dimension F$ (resp.\ $\codimension F$).

\paragraph For any $F \in D^b(X)$, we will write $F^\vee$ to denote the derived dual $R\HHom (F,\OO_X)$ of $F$.  When $F$ is a coherent sheaf of codimension $c$, we will write $F^\ast$ to denote the usual sheaf dual of $F$, i.e.\ $\EExt^c(F,\OO_X)$; note that $H^c (F^\vee)= F^\ast$.  Given a pure codimension-$c$ coherent sheaf $F$ on $X$, we will say $F$ is reflexive if the natural injection $F \hookrightarrow F^{\ast \ast}$ is an isomorphism.

\paragraph Recall that the homological dimension  of a coherent sheaf $F$ on a smooth projective variety $X$ is defined to be the minimal length of a locally free resolution of $F$, and that a coherent sheaf of homological dimension $n$ satisfies $\EExt^i (F,\OO_X) =0$ for all $i>n$ and hence $F^\vee \in D^{[0,n]}_{\Coh (X)}$.  We will write $\mathrm{hd}(F)$ to denote the homological dimension of a coherent sheaf $F$ on $X$.


\paragraph[Stable pairs] On \label{para:PTstablepair} a smooth projective threefold $X$, a stable pair  in the sense of Pandharipande-Thomas \cite{PT} is a pure 1-dimensional sheaf $F$ together with a section $\OO_X \overset{s}{\to} F$ such that $\cokernel (s)$ is 0-dimensional.  The purity of $F$ implies that the support of $F$ is a Cohen-Macaulay curve.  We often think of a stable pair $\OO_X \overset{s}{\to} F$ as a 2-term complex representing an object in $D^b(X)$, with $F$ sitting at degree 0.  We will refer to a stable pair in the sense of \cite{PT} as a PT stable pair, or simply a stable pair.

\paragraph[PT stable objects] Bayer \label{para:PTstabToda} characterised PT stable pairs using the notion of polynomial stability in \cite{BayerPBSC}.  There is a particular polynomial stability $\sigma_{PT}$ on $D^b(X)$, referred to as PT stability by Bayer, such that the $\sigma_{PT}$-stable objects objects $E$  in the heart
\[
  \Ac^p := \langle \Coh^{\geq 2}(X)[1], \Coh^{\leq 1}(X)\rangle
\]
with $\ch_0 (E) = -1, \ch_1 (E)=0$ and $\det{E} = \OO_X$ are precisely the PT stable pairs  in \ref{para:PTstablepair}.  We will refer to $\sigma_{PT}$-(semi)stable objects in $\Ac^p$ of any Chern character as PT (semi)stable objects.   The properties of higher-rank PT stable objects and their moduli spaces were studied  in \cite{Lo1,Lo2}.

\subparagraph Suppose \label{para:PTstabobjproperties} $X$ is a smooth projective threefold, and $\omega$ is a fixed ample class on $X$ that appears in the definition of PT stability.  Then every PT semistable object $E$ with nonzero $\ch_0$ satisfies the following properties:
\begin{itemize}
\item[(i)] $H^{-1}(E)$ is torsion-free and $\mu_\omega$-semistable,
\item[(ii)] $H^0(E)$ is 0-dimensional,
\item[(iii)] $\Hom_{D^b(X)}(\Coh^{\leq 0}(X),E)=0$;
\end{itemize}
moreover, when $\ch_0(E)$ and $\omega^2 \ch_1(E)$ are coprime, every object in $\Ac^p$ satisfying (i) through (iii) is a PT stable object, and PT stability coincides with PT semistability \cite[Proposition 2.24]{Lo2}.  Also, properties (i) and (ii) implies that, if $E$ is a PT-semistable object, then $\ch_0(E) = -n$ for some nonnegative integer $n$;  we will sometimes refer to such an $E$ as a rank $n$ PT semistable object by abuse of notation.

\subparagraph  Under derived dual and up to a shift, PT stability corresponds to another polynomial stability $\sigma_{PT}^\ast$, meaning  $\sigma_{PT}$-stable objects and $\sigma_{PT}^\ast$-stable objects correspond to each other via derived dual.  We will refer to the $\sigma_{PT}^\ast$-(semi)stable objects as dual-PT (semi)stable objects; their properties and moduli spaces were studied in \cite{Lo4}.

\subparagraph  Suppose \label{para:dualPTstabobjproperties} $X$ is a smooth projective threefold, and $\omega$ is a fixed ample class on $X$ that appears in the definition of dual-PT stability.  Then a standard argument shows that every dual-PT semistable object $E$ with nonzero $\ch_0$ satisfies the following properties besides lying in $\Ac^p$:
\begin{itemize}
\item[(i)] $H^{-1}(E)$ is torsion-free and $\mu_\omega$-semistable.
\item[(ii)] $\Hom_{D^b(X)}(\Coh^{\leq 1}(X),E)=0$.
\end{itemize}
Property (ii) implies that $H^{-1}(E)$ is a reflexive sheaf.  Also, when $\ch_0(E)$ and $\omega^2 \ch_1(E)$ are coprime, every object in $\Ac^p$ satisfying (i) and (ii) is a dual-PT stable object, and dual-PT stability coincides with dual-PT semistability \cite[Lemma 3.5]{Lo4}.

\begin{rem}
In Toda's work \cite{Toda2}, he directly defines PT semistable objects  to be the objects in $D^b(X)$ satisfying properties (i) through (iii).   All the computations in \cite{Toda2}, however, are performed under the assumption that $\ch_0, \omega^2 \ch_1$ are coprime; under this assumption, the  PT semistable objects  Toda studies coincide with the PT semistable objects defined using  Bayer's polynomial stability  (as in \ref{para:PTstabToda}).
\end{rem}

\section{The dualising functor}

In this section, we study the behaviour of a class of  2-term complexes under the derived dual functor $^\vee$.  These 2-term complexes can be taken to be various stable objects (see Section \ref{sec:esssurj}) and, in particular, PT stable objects.

\begin{lem}\label{lem:Eveeleq3Coh0vanishing}
Let $E$ be an object of $D^b(X)$  satisfying $E^\vee \in D^{\leq 3}_{\Coh (X)}$.  Then
\[
  \Hom_{D^b(X)}(\Coh^{\leq 0}(X),E)=0 \text{ if and only if } H^3(E^\vee)=0.
\]
\end{lem}

\begin{proof}
For any $E \in D^b(X)$ and $T \in \Coh^{\leq 0}(X)$ we have
\[
  \Hom (T,E) \cong \Hom  (E^\vee , T^\ast [-3]).
\]
Therefore, when $E$ satisfies $E^\vee \in D^{\leq 3}_{\Coh (X)}$, we have $H^3 (E^\vee)=0$ if and only if $\Hom (T,E)=0$ for all $T \in \Coh^{\leq 0}(X)$, i.e.\ $\Hom (\Coh^{\leq 0}(X), E)=0$.
\end{proof}

\begin{eg}\label{rem:HomCohleq0H3}
 For   any $E \in \langle \Coh^{\geq 1}(X)[1], \Coh^{\leq 1}(X)\rangle$, in the associated exact triangle
\[
  H^0(E)^\vee \to E^\vee \to H^{-1}(E)^\vee [-1] \to H^0(E)^\vee [1]
\]
we have $H^0(E)^\vee \in D^{[2,3]}_{\Coh (X)}$ and $H^{-1}(E)^\vee \in D^{[0,2]}_{\Coh (X)}$ \cite[Proposition 1.1.6]{HL} and hence $E^\vee \in D^{\leq 3}_{\Coh (X)}$.  Thus
\[
\Hom (\Coh^{\leq 0}(X),E)=0 \text{ if and only if } H^3(E^\vee)=0
\]
for such $E$ by Lemma \ref{lem:Eveeleq3Coh0vanishing}.
\end{eg}

\begin{lem}\label{lem:vanishingCoh01}
Suppose $X$ is a smooth projective threefold and $F \in \Coh^{\geq 2}(X)$.  Then
\begin{itemize}
\item[(i)] $F$ has homological dimension at most 1 if and only if
\begin{equation}\label{eq:hd1cond}
  \Hom (\Coh^{\leq 0}(X), F [1]) =0.
\end{equation}
\item[(ii)] If $F$ is torsion-free, then $F$ is reflexive  if and only if
\begin{equation}\label{eq:reflcond}
  \Hom (\Coh^{\leq 1}(X), F[1])=0.
\end{equation}
\end{itemize}
\end{lem}

\begin{proof}
(i) Taking derived dual, we observe that
\begin{equation}\label{eq:AG44-85-lem3-1}
\Hom (\Coh^{\leq 0}(X), F[1])=0 \text{\quad  if and only if \quad} \Hom (F^\vee, \Coh^{\leq 0}(X) [-2])=0
\end{equation}
for any $F \in D^b(X)$.   For   any $F \in \Coh^{\geq 2}(X)$, we know  $F^\vee \in D^{[0,2]}_{\Coh (X)}$ from \cite[Proposition 1.1.6]{HL}, and so $F$ satisfies  the equivalent conditions in \eqref{eq:AG44-85-lem3-1} if and only if $F^\vee \in D^{[0,1]}_{\Coh (X)}$, which in turn is equivalent to $F$ having homological dimension at most 1 \cite[III.6]{Harts}.


(ii) This is a special case of \cite[Lemma 4.20]{CL}.
\end{proof}



The results in the remainder of this  provide a  common ground across the  constructions in this article,  Toda's work \cite{Toda2},    Gholampour-Kool's work \cite{gholampour2017higher}, and the author's previous work \cite{Lo4}.

\begin{lem}\label{coro:AG44-84-1}
The category
\begin{equation}\label{AG44-84-last}
  \{ E \in \langle \Coh^{=3}(X)[1],\Coh^{\leq 1}(X)\rangle : \Hom (\Coh^{\leq 0}(X),E)=0 \}
\end{equation}
is invariant under the functor $(-^\vee)[2]$.  Moreover, $H^{-1}(E)$ is slope (semi-)stable if and only if $H^{-1}(E^\vee [2])$ is so.
\end{lem}

\begin{proof}
Take any object $E$ in the  category \eqref{AG44-84-last}. If $H^{-1}(E)=0$, then $E=H^0(E)$ is a pure 1-dimensional sheaf, in which case $E^\vee [2] \cong \EExt^2(E,\OO_X)$ is also a pure 1-dimensional sheaf and hence again lies in the category \eqref{AG44-84-last}.  So let us suppose $H^{-1}(E) \neq 0$ from now on.  The exact triangle
\[
  H^0(E)^\vee \to E^\vee \to H^{-1}(E)^\vee [-1] \to H^0(E)^\vee [1]
\]
gives the long exact sequence
\begin{equation*}
\xymatrix{
0 \ar[r] & H^1(E^\vee) \ar[r] & H^{-1}(E)^\ast \ar `d[ll] `[dll] [dll] \\
\EExt^2 (H^0(E),\OO_X) \ar[r] & H^2(E^\vee) \ar[r] & \EExt^1 (H^{-1}(E),\OO_X) \ar `d[ll] `[dll] [dll] \\
\EExt^3 (H^0(E),\OO_X) \ar[r] & H^3 (E^\vee) \ar[r] & \EExt^2 (H^{-1}(E),\OO_X) \ar `d[ll] `[dll] [dll] \\
0  \ar[r]  & \cdots &
}
\end{equation*}
from which we see $H^{-1}(E^\vee [2]) = H^1(E^\vee)$ is a subsheaf of a torsion-free sheaf, and hence is torsion-free.  We also have $H^3(E^\vee)=0$ by  Example \ref{rem:HomCohleq0H3}.  Since $H^{-1}(E)$ is torsion-free, it follows that $\EExt^1 (H^{-1}(E),\OO_X) \in \Coh^{\leq 1}(X)$ \cite[Proposition 1.1.6 ii)]{HL}.  On the other hand, since $H^0(E) \in \Coh^{\leq 1}(X)$, we have  $\EExt^2 (H^0(E),\OO_X) \in \Coh^{\leq 1}(X)$.  Hence $H^2(E^\vee)=H^0(E^\vee [2]) \in \Coh^{\leq 1}(X)$, giving us   $E^\vee [2] \in \langle \Coh^{=3}(X)[1],\Coh^{\leq 1}(X)\rangle$ overall.

Let us  write $F = E^\vee [2]$.  Then  $H^3(F^\vee) = H^3(E[-2])=0$, and so $\Hom (\Coh^{\leq 0}(X),F)=0$ by Example \ref{rem:HomCohleq0H3}, i.e.\ $F$ lies in \eqref{AG44-84-last}.

Lastly, if $H^{-1}(E)$ is slope (semi-)stable then so is its dual $H^{-1}(E)^\ast$; since $\EExt^2 (H^0(E),\OO_X)$ has codimension at least 2, this means that  $H^1(E^\vee)=H^{-1}(F)$ is also slope (semi-)stable.  On the other hand, if $H^{-1}(F)=H^1(E^\vee)$ is slope (semi-)stable, then so is $H^{-1}(E)^\ast$ from the long exact sequence above, implying $H^{-1}(E)^{\ast \ast}$ is also slope (semi-)stable.  Since $H^{-1}(E)$ and $H^{-1}(E)^{\ast \ast}$ are isomorphic in codimension 1, this means that $H^{-1}(E)$ itself is slope (semi-)stable.
\end{proof}

\paragraph The category $\Coh^{=1}(X)$ of pure 1-dimensional coherent sheaves on $X$ is  invariant under the functor $(-^\vee)[2]$ by \cite[p.6]{HL}.  As a result, for any $E \in D^{\leq 0}_{\Coh (X)}$ and $T \in \Coh^{=1}(X)$, we have the isomorphisms
\begin{equation}\label{eq:AG44-128-1}
  \Hom (T,E^\vee [2]) \cong \Hom (E,T^\vee [2]) \cong \Hom (E,T^\ast) \cong \Hom (H^0(E),T^\ast)
\end{equation}
where the first isomorphism follows from taking derived dual, and the last isomorphism uses the assumption $E \in D^{\leq 0}_{\Coh (X)}$.

\begin{lem}\label{lem:AG44-127-star}
The functor $(-^\vee)[2]$ induces an equivalence of subcategories of \eqref{AG44-84-last}
    \begin{multline}
      \{ E \in \langle \Coh^{=3}(X)[1], \Coh^{\leq 0}(X)\rangle : \Hom (\Coh^{\leq 0}(X),E)=0\} \overset{\thicksim}{\to}  \\
      \{ E \in \langle \Coh^{=3}(X)[1], \Coh^{\leq 1}(X)\rangle: \Hom (\Coh^{\leq 1}(X), E) =0 \}. \label{eq:AG44-124-5}
    \end{multline}
\end{lem}

All the PT stable objects (resp.\ dual-PT stable objects) lie in the left-hand side (resp.\ right-hand side) of \eqref{eq:AG44-124-5}.  As a result, Lemma \ref{lem:AG44-127-star} can be considered as a purely homological version of the statement that `PT stable objects and dual-PT stable objects correspond to each other under derived dual'.  We also note that this Lemma had essentially appeared in Piyaratne-Toda's work \cite[Lemma 4.16]{piyaratne2015moduli} in their study of the moduli spaces of Bridgeland semistable objects on threefolds.

Also, by Lemma \ref{lem:vanishingCoh01}, all the  torsion-free coherent sheaves of homological dimension at most 1 (resp.\ torsion-free reflexive sheaves) sitting at degree $-1$  lie in the left-hand side (resp.\ right-hand side) of \eqref{eq:AG44-124-5}.

\begin{proof}
Suppose $E$ is an object in the left-hand side of \eqref{eq:AG44-124-5}.  By Lemma \ref{coro:AG44-84-1}, it suffices to show that $\Hom (\Coh^{=1}(X),E^\vee [2])=0$.  For any $T \in \Coh^{=1}(X)$, we have $\Hom (T,E^\vee [2]) \cong \Hom (H^0(E),T^\ast)$ by \eqref{eq:AG44-128-1};  the latter Hom  vanishes since $H^0(E)\in \Coh^{\leq 0}(X)$ by assumption while  $T^\ast$ is pure 1-dimensional.

Conversely, suppose $E$ is an object in the right-hand side of \eqref{eq:AG44-124-5}.   Again, by Lemma \ref{coro:AG44-84-1}, it suffices to show that $H^0(E^\vee [2]) \in \Coh^{\leq 0}(X)$.  Since Lemma \ref{coro:AG44-84-1} already gives  $H^0(E^\vee [2])\in \Coh^{\leq 1}(X)$, it suffices to show  $\Hom (H^0(E^\vee [2]), \Coh^{=1}(X))=0$.  To this end, take any $T \in \Coh^{=1}(X)$; note that $T$ is reflexive.  Then  \eqref{eq:AG44-128-1} gives $\Hom (T^\ast, E) \cong \Hom (H^0(E^\vee [2]),T)$, in which the former Hom  vanishes because $\Hom (\Coh^{\leq 1}(X),E)=0$ by assumption.
\end{proof}


\begin{cor}\label{cor:2}
The category
\begin{equation}\label{eq:invcat2}
  \{ E \in \langle \Coh^{=3}(X)[1], \Coh^{\leq 0}(X)\rangle : \Hom (\Coh^{\leq 1}(X), E)=0\}
\end{equation}
is invariant under the functor $-^\vee [2]$.
\end{cor}

Note that the category \eqref{eq:invcat2} contains all the  reflexive sheaves (shifted by 1); in fact, it contains all the objects that are PT semistable and dual-PT semistable at the same time, the moduli space of which was studied in \cite[Theorem 1.2]{Lo4} with a coprime assumption on the rank and degree of the objects.

\section{A functor taking objects to morphisms}

In this section, we construct a functor $\wt{\Fc}$ that takes a subcategory $\EE$ of  $E \in D^{[-1,0]}_{\Coh (X)}$ into the category $\Mor (D^b(X))$ (Proposition \ref{prop:Fissurjfunctor}).  Composing with the cohomology functor $H^0$, we  obtain a functor from $\EE$ into $\Mor (\Coh (X))$ (see \ref{para:wtFcwithH0}).

\paragraph We define the full subcategory of $D^b(X)$
\[
\EE  = \{ E \in D^{[-1,0]}_{\Coh (X)} : \mathrm{hd}(H^{-1}(E))\leq 1, H^0(E) \in \Coh^{\leq 0}(X)\}
\]
and  the full subcategory of $\Mor (D^b(X))$
\begin{equation*}\label{eq:obinHc}
  \Lc = \{ A^\vee [1] \overset{m}{\to} B : A \in \Coh (X), \mathrm{hd}(A) \leq 1, B \in \Coh^{\leq 0}(X)\}.
\end{equation*}

\subparagraph Note \label{para:hdleqmeansCohgeq2} that for any coherent sheaf $F$ on $X$, the condition $\mathrm{hd}(F) \leq 1$ implies $\Coh^{\geq 2}(X)$ by \cite[Proposition 1.1.6]{HL}.  Hence $H^{-1}(E) \in \Coh^{\geq 2}(X)$ for any $E \in \EE$.

\paragraph For \label{para:defFcfunctor} any $E \in \EE$, truncation functors with respect to the standard t-structure on $D^b(X)$ give  the canonical exact triangle
\begin{equation}\label{eq:Fdef-eq1}
  H^{-1}(E) [1] \to E \to H^0(E) \overset{w}{\to} H^{-1}(E)[2].
\end{equation}
Applying the derived dual functor to \eqref{eq:Fdef-eq1} followed by the shift functor $[2]$ gives the exact triangle
\begin{equation}\label{eq:Fdef-eq2}
H^0(E)^\vee[2] \to E^\vee[2] \to H^{-1}(E)^\vee [1] \overset{w^\vee[3]}{\xrightarrow{\hspace{0.8cm}}} H^0(E)^\vee [3]
\end{equation}
where $H^{-1}(E)^\vee [1] \in D^{[-1,0]}_{\Coh (X)}$ by the assumption $\mathrm{hd}(H^{-1}(E))\leq 1$, and $H^0(E)^\vee [3] \in \Coh^{\leq 0} (X)$.  That is, the morphism $w^\vee [3]$ is an object of $\Lc$.

\begin{defn}
For any $E \in \EE$, we define $\wt{\Fc} (E)$ to be the object $w^\vee [3]$ of $\Lc$ in the notation of \ref{para:defFcfunctor}.
\end{defn}

\begin{prop}\label{prop:Fissurjfunctor}
$\wt{\Fc}$ is an essentially surjective (contravariant) functor from $\EE$ to $\Lc$, and induces a bijection between the isomorphism classes in $\EE$ and $\Lc$.
\end{prop}

\begin{proof}
  \textbf{Step 1.} Given any morphism $E_1 \overset{f}{\to} E_2$ in $\EE$, the truncation functors give a morphism of exact triangles in $D^b(X)$
  \[
  \xymatrix{
  H^{-1}(E_1)[1] \ar[r] \ar[d] & E_1 \ar[r] \ar[d]^f & H^0(E_1) \ar[r]^(.45){w_1} \ar[d]^{H^0(f)} & H^{-1}(E_1)[2] \ar[d]^{H^{-1}(f)[2]} \\
  H^{-1}(E_2)[1] \ar[r] & E_2 \ar[r] & H^0(E_2) \ar[r]^(.45){w_2} & H^{-1}(E_2)[2]
}.
\]
Applying the derived dual functor followed by  the shift functor $[2]$ gives
\[
\xymatrix{
  H^0(E_2)^\vee[2] \ar[r] \ar[d]  & E_2^\vee[2] \ar[r] \ar[d]^{f^\vee[2]} & H^{-1}(E_2)^\vee [1] \ar[r]^{w_2^\vee [3]}  \ar[d] &   H^0(E_2)^\vee [3] \ar[d] \\
  H^0(E_1)^\vee[2] \ar[r]  & E_1^\vee[2] \ar[r]  & H^{-1}(E_1)^\vee [1] \ar[r]^{w_1^\vee [3]}  & H^0(E_1)^\vee [3]
}.
\]
The right-most square now gives a morphism from $w_1^\vee [3]$ to $w_2^\vee [3]$ in $\Lc$.  It is clear that $\wt{\Fc}$ respects composition of morphisms in $\EE$, and so $\wt{\Fc}$ is a functor from $\EE$ to $\Lc$.  

\textbf{Step 2.} To show the essential surjectivity of $\wt{\Fc}$, let us take an arbitrary  element of $\Lc$, say the diagram in $D^b(X)$
\[
  A^\vee [1] \overset{m}{\to} B
\]
where $A$ is a  sheaf of homological dimension at most 1, and $B$ is a sheaf supported in dimension 0.   We first complete $m$ to an exact triangle in $D^b(X)$
\begin{equation}\label{eq:Scobjcone-eq1}
A^\vee [1] \overset{m}{\to} B \to C \to A^\vee [2].
\end{equation}
Applying $[-3]^\vee$ now gives the exact triangle
\[
  A[1] \to G \to B^\vee [3] \overset{m^\vee [3]}{\longrightarrow} A[2]
\]
where  $G := C^\vee [3]$ is an object of $\EE$.  Since $A[1] \in D^{\leq -1}_{\Coh (X)}$ and $B^\vee [3] \in D^{\geq 0}_{\Coh (X)}$, there is a canonical isomorphism of exact triangles in $D^b(X)$ \cite[Lemma 5, IV.4]{GM}
\begin{equation}\label{eq:Scobjcone-eq2}
\xymatrix{
  A[1] \ar[r] \ar[d] & G \ar[r] \ar[d] & B^\vee [3] \ar[r]^{m^\vee [3]} \ar[d] & A[2] \ar[d] \\
  H^{-1}(G)[1] \ar[r] & G \ar[r] & H^0(G) \ar[r]^(.4)w & H^{-1}(G) [2]
}.
\end{equation}
Applying the functor $^\vee[2]$ gives the isomorphism of exact triangles
\begin{equation}\label{eq:Ffunctorpreimageimage}
\xymatrix{
  H^0(G)^\vee [2] \ar[r] \ar[d] & G^\vee [2] \ar[r] \ar[d] & H^{-1}(G)^\vee [1] \ar[r]^{w^\vee [3]} \ar[d] & H^0(G)^\vee [3] \ar[d] \\
  B[-1] \ar[r] & G^\vee [2] \ar[r] & A^\vee [1] \ar[r]^m &
  B
}.
\end{equation}
Since $w^\vee [3]$ is precisely $\wt{\Fc} (G)$, we have shown the essential surjectivity of $\wt{\Fc}$.

\textbf{Step 3.} To show that $\wt{\Fc}$ induces a bijection between the isomorphism classes in $\EE$ and $\Lc$, let us take two objects $E_1, E_2$ in $\EE$ and suppose there is an isomorphism from $\wt{\Fc} (E_2)$ to  $\wt{\Fc} (E_1)$ in $\Lc$, say given by the diagram in $D^b(X)$
\[
\xymatrix{
  H^{-1}(E_2)^\vee [1] \ar[r] \ar[d]^h & H^0(E_2)^\vee [3] \ar[d]^i\\
  H^{-1}(E_1)^\vee [1] \ar[r] & H^0(E_1)^\vee [3]
}
\]
where $h,i$ are isomorphisms.  We can complete the rows of this square to  exact triangles of the form \eqref{eq:Fdef-eq2}
\[
\xymatrix{
 H^0(E_2)^\vee [2] \ar[r] \ar[d]^{i[-1]} & E_2^\vee [2] \ar[r] \ar@{.>}[d]^g & H^{-1}(E_2)^\vee [1] \ar[r] \ar[d]^h & H^0(E_2)^\vee [3] \ar[d]^i\\
 H^0(E_1)^\vee [2] \ar[r] & E_1^\vee [2] \ar[r] & H^{-1}(E_1)^\vee [1] \ar[r] & H^0(E_1)^\vee [3]
};
\]
and then $h, i$ can be completed with an isomorphism $g : E_2^\vee [2]\to E_1^\vee [2]$ to an isomorphism of exact triangles \cite[Corollary 4a, IV.1]{GM}.  Hence $E_1$ and $E_2$ are isomorphic in $D^b(X)$.  This completes the proof of the proposition.

\end{proof}

\subparagraph In the proof of essential surjectivity in Proposition \ref{prop:Fissurjfunctor} (i.e.\ Step 2 of the proof), it is not clear that the construction  taking the object $m$ in $\Lc$ to the object $G$ in $\EE$ is a functor, since the object $C$ is defined up to an isomorphism that is not necessarily canonical. 

\subparagraph The \label{para:wtFcwithH0} degree-zero cohomology functor with respect to the standard t-structure  $H^0 : D^b(X) \to \Coh (X)$ induces a functor $\Mor (D^b(X)) \to \Mor (\Coh (X))$ which we will also denote by $H^0$.  For any object in $\Lc$ of the form
\[
  A^\vee [1] \overset{m}{\to} B,
\]
  its image under $H^0$ is
\[
  H^0(A^\vee [1]) = \EExt^1 (A,\OO_X) \xrightarrow{H^0(m)} B.
\]
Let us use the notation in the proof of Proposition \ref{prop:Fissurjfunctor} and take $C, w$   as in \eqref{eq:Scobjcone-eq1}, \eqref{eq:Scobjcone-eq2}, respectively.  Then from \eqref{eq:Ffunctorpreimageimage}, the morphism of sheaves $H^0(m)$ is surjective if and only if the morphism of sheaves $H^0(w^\vee [3])$ is surjective.

\section{Essential surjectivity of the functor $\Fc$}\label{sec:esssurj}

In this section, we will modify the functor $\wt{\Fc}$ to a new functor $\Fc$, and show  that each object $E \in D^{[-1,0]}_{\Coh (X)}$ of the following types is taken by $\Fc$ to a surjective morphism of sheaves $\EExt^1 (H^{-1}(E),\OO_X) \to H^0(E)^\ast$: PT-semistable objects, dual-PT semistable objects, objects giving rise to $L$-invariants in the sense of Toda \cite{Toda2}, and stable frozen triples in the sense of Sheshmani \cite{sheshmani2016higher}.  In particular, under a coprime assumption on rank and degree, we prove in Theorem \ref{thm:main1} that $\Fc$ restricts to an essentially surjective functor from the category of PT stable objects to a category of surjective morphisms of coherent sheaves.


\begin{lem}\label{lem:surjiffCohleq0vanish}
Suppose $E$ is an object in $\langle \Coh^{\geq 2}(X)[1], \Coh^{\leq 1}(X)\rangle$ and $H^{-1}(E)$ has homological dimension at most 1.  Suppose
\begin{equation}
  H^{-1}(E)[1] \to E \to H^0(E) \overset{w}{\to} H^{-1}(E)[2] \tag{\ref{eq:Fdef-eq1}}
\end{equation}
is the associated canonical exact triangle in $D^b(X)$.  Then
\[
  H^0\wt{\Fc} (E) = H^0(w^\vee [3]) : \EExt^1 (H^{-1}(E),\OO_X) \to H^0(E)^\ast
\]
is a surjection in $\Coh (X)$ if and only if $\Hom_{D^b(X)}(\Coh^{\leq 0}(X),E)=0$.
\end{lem}

\begin{proof}
Applying $^\vee[2]$ to \eqref{eq:Fdef-eq1} gives us the exact triangle
\begin{equation}
H^0(E)^\vee[2] \to E^\vee[2] \to H^{-1}(E)^\vee [1] \overset{w^\vee[3]}{\xrightarrow{\hspace{0.8cm}}} H^0(E)^\vee [3] \tag{\ref{eq:Fdef-eq2}},
\end{equation}
which has long exact sequence of cohomology
\begin{equation}\label{eq:AG44-144-1}
\cdots \to \EExt^1 (H^{-1}(E),\OO_X) \overset{H^0(w^\vee [3])}{\xrightarrow{\hspace{1cm}}} H^0(E)^\ast \to H^1(E^\vee [2]) \to \EExt^2 (H^{-1}(E),\OO_X)
\end{equation}
where the last term $\EExt^2 (H^{-1}(E),\OO_X)$ vanishes since $H^{-1}(E)$ has homological dimension at most 1.  Therefore, the morphism $H^0\wt{\Fc} (E) =H^0(w^\vee [3])$ is surjective if and only if $H^3(E^\vee)=H^1(E^\vee [2])=0$.  On the other hand, $H^3 (E^\vee)=0$ if and only if $\Hom (\Coh^{\leq 0}(X),E)=0$ by Example \ref{rem:HomCohleq0H3}, and so we are done.
\end{proof}

\begin{eg}\label{eg:objformainthm}
Suppose $E$ is an object in $\langle \Coh^{=3}(X)[1], \Coh^{\leq 1}(X)\rangle$ that satisfies the vanishing
\[
  \Hom (\Coh^{\leq 0}(X),E)=0.
\]
Note that the category $\Ac:= \langle \Coh^{\geq 2}(X)[1], \Coh^{\leq 1}(X) \rangle$  is the heart of a t-structure on $D^b(X)$  and hence an abelian category \cite[Section 3]{BayerPBSC}.  Also, the subcategory $\Coh^{\leq 0}(X)$ is closed under quotients in $\Ac$.  Hence the vanishing $\Hom (\Coh^{\leq 0}(X),E)=0$ implies $\Hom (\Coh^{\leq 0}(X), H^{-1}(E)[1])=0$, which in turn implies  $H^{-1}(E)$ has homological dimension at most 1 by Lemma \ref{lem:vanishingCoh01}(i).  That is, $E$ satisfies the hypotheses of Lemma \ref{lem:surjiffCohleq0vanish}.  As a result, all of the following objects satisfy the hypotheses of Lemma \ref{lem:surjiffCohleq0vanish} in addition to the vanishing $\Hom (\Coh^{\leq 0}(X),-)=0$:
\begin{itemize}
\item[(a)] Objects  in the left-hand side of the equivalence \eqref{eq:AG44-124-5}, i.e.\ in the category
    \[
     \{ E \in \langle \Coh^{=3}(X)[1], \Coh^{\leq 0}(X)\rangle : \Hom (\Coh^{\leq 0}(X),E)=0\}.
    \]
    These include all the PT-semistable objects (see \ref{para:PTstabobjproperties}).
\item[(b)] Objects  in the right-hand side of the equivalence \eqref{eq:AG44-124-5}, i.e.\ in the category
    \[
      \{ E \in \langle \Coh^{=3}(X)[1], \Coh^{\leq 1}(X)\rangle: \Hom (\Coh^{\leq 1}(X), E) =0 \}.
    \]
    These include all the dual-PT semistable objects (see \ref{para:dualPTstabobjproperties}).
\item[(c)] Objects  in
\[
\{ E \in \langle \Coh_\mu (X), \mathcal{C}_{[0,\infty]} \rangle : \Hom ( \mathcal{C}_{[0,\infty]},E)=0\},
\]
which are  the objects giving rise to the $L$-invariants defined by Toda in proving a   higher-rank DT/PT correspondence in \cite{Toda2}.  Here, we have some fixed ample class $\omega$ on $X$, and $\Coh_\mu (X)$ is the category of all $\mu_\omega$-semistable coherent sheaves $E$ with $\mu_\omega (E) := \omega^2 \ch_1(E)/\ch_0(E) = \mu$.  On the other hand, the category $\mathcal{C}_{[0,\infty]}$ consists of coherent sheaves $F$ supported in dimension at most 1, such that all its  Harder-Narasimhan factors with respect to the slope function $\ch_3(-)/\omega \ch_2(-)$ have slopes lying in the interval $[0,\infty]$.
\end{itemize}
\end{eg}

\paragraph Lemma \ref{lem:surjiffCohleq0vanish} motivates us to define the full subcategory of $\EE$
\[
  \EE_0 = \{ E \in \EE : \Hom_{D^b(X)}(\Coh^{\leq 0}(X),E)=0 \}.
\]
For any coherent sheaf $F$ on $X$, we will also define the full subcategory of   $\Mor (\Coh (X))$
\[
\Sc (F) = \{ F \overset{q}{\to} Q : q \text{ is a surjection in $\Coh (X)$}, Q \in \Coh^{\leq 0}(X)\}.
\]
 For  any $E \in \EE_0$, we have $H^{-1}(E) \in \Coh^{\geq 2}(X)$ by \ref{para:hdleqmeansCohgeq2}, and so $E$ satisfies the hypotheses of Lemma \ref{lem:surjiffCohleq0vanish}.  Since $H^0(E)$ is 0-dimensional, we obtain that   $H^0 \wt{\Fc} (E)$ lies in  the category $\Sc (\EExt^1 (H^{-1}(E),\OO_X))$.  Note that $H^{-1}(E)$ has homological dimension at most 1 from the definition of $\EE$.  This allows us to make the following definition:

\begin{defn} \label{para:mainthmprep1}
We write $\Fc$ to denote the  restriction of $\wt{\Fc}$
\begin{equation}\label{eq:H0Fcfuncrestriction}
  \Fc = (H^0 \circ \wt{\Fc}) |_{\EE_0} : \EE_0 \to \coprod_{\substack{F \in \Coh (X), \, \mathrm{hd}(F)\leq 1}} \Sc (\EExt^1(F,\OO_X)).
\end{equation}
\end{defn}



\begin{eg}[PT semistable objects]\label{eg:PTssinEE0}
Every PT semistable object $E$ of nonzero $\ch_0$ lies in $\EE_0$.  To see this, note that the canonical exact triangle $H^{-1}(E)[1] \to E \to H^0(E) \to H^{-1}(E)[2]$ gives a short exact sequence $0 \to H^{-1}(E)[1] \to E \to H^0(E) \to 0$ in $\Ac^p$.  Since $\Coh^{\leq 0}(X)$ is closed under quotient in $\Ap$, the vanishing $\Hom (\Coh^{\leq 0}(X),E)=0$ implies the vanishing $\Hom (\Coh^{\leq 0}(X), H^{-1}(E)[1])=0$, and so $\mathrm{hd}(H^{-1}(E))\leq 1$ by Lemma \ref{lem:vanishingCoh01}(i).  The claim then follows from \ref{para:PTstabobjproperties}
\end{eg}

\begin{eg}[Sheshmani's frozen triples]\label{eg:stabfrozentripleinEE0}
Every stable frozen triple on a smooth projective threefold $X$ in the sense of Sheshmani \cite{sheshmani2016higher} represents an object in $\EE_0$.  A frozen triple $(G,F,\varphi)$ consists of a locally free sheaf $G \cong \OO_X^{\oplus r}$ on $X$ (for some positive integer $r$) together with a morphism of coherent sheaves $\varphi : G \to F$ where $F$ is pure 1-dimensional.  Such a frozen triple is called stable (or $\tau'$-stable in \cite{sheshmani2016higher}) if the cokernel of $\varphi$ is 0-dimensional.  Let us write $E$ to represent the 2-term complex $[G \overset{\varphi}{\to} F]$ in $D^b(X)$ with $F$ sitting at degree 0.  Then clearly $E \in \Ap$.  For any $T \in \Coh^{\leq 0}(X)$, applying $\Hom (T,-)$ to the  exact triangle in $D^b(X)$
\[
 G\overset{\varphi}{\to} F \to E \to G[1]
\]
gives $\Hom (T,E)=0$ and so $\Hom (\Coh^{\leq 0}(X),E)=0$.  The same argument as in Example \ref{eg:PTssinEE0} then shows $\mathrm{hd}(H^{-1}(E))\leq 1$, and so $E \in \EE_0$.  Note that this example is already implicitly stated in \cite[Example 3.2]{Toda2}.
\end{eg}

\begin{lem}\label{lem:AG44-97-lem1}
Given any $A \in\Coh (X)$ with $\mathrm{hd}(A)\leq 1$ and a morphism of sheaves $r : \EExt^1 (A,\OO_X) \to Q$ where $Q \in \Coh^{\leq 0}(X)$, there exists an object $G$ in $\EE$ such that $H^0 \wt{\Fc} (G)$ is isomorphic to $r$ in  $\Mor(\Coh (X))$.  Moreover, if $r$ is a surjection in $\Coh (X)$, then we can take $G$ to be in $\EE_0$.
\end{lem}

\begin{proof}
Consider the composition of morphisms in $D^b(X)$
\[
  A^\vee [1] \overset{c}{\to} H^0(A^\vee [1]) = \EExt^1 (A,\OO_X) \overset{r}{\to} Q
\]
where $c$ is canonical.  Since $rc \in \Lc$,   the essential surjectivity of $\wt{\Fc}$ from Proposition \ref{prop:Fissurjfunctor}  implies that  there exists some $G \in \EE$ such that $\wt{\Fc} (G) \cong rc$ in $\Lc$, which in turn implies  $r \cong H^0\wt{\Fc} (G)$ in $\Mor(\Coh (X))$.

In fact, from the construction of the object $G$ (see Step 2 of the proof of Proposition \ref{prop:Fissurjfunctor}), we have the exact triangle in $D^b(X)$
\[
  A[1] \to G \to Q^\vee [3] \to A[2]
\]
which gives $H^{-1}(G) \cong A$ and $H^0(G) \cong Q^\ast$, i.e.\ $G$ satisfies the hypotheses of Lemma \ref{lem:surjiffCohleq0vanish}.  Therefore, if $r$ is a surjection in $\Coh (X)$,  the morphism of sheaves $H^0\wt{\Fc} (G)$ is also a surjection in $\Coh (X)$, and Lemma \ref{lem:surjiffCohleq0vanish} implies that $\Hom (\Coh^{\leq 0}(X),G)=0$, meaning $G \in \EE_0$.
\end{proof}

\paragraph We \label{para:Fcrestrictions} have the following restriction of $\Fc$:
\begin{align*}
  \Fc_{\mathrm{tf}} : \{ E \in \EE_0 : H^{-1}(E) \text{ is torsion free }\} &\to \coprod_{\substack{A \in \Coh^{=3}(X), \, \mathrm{hd}(A)\leq 1}} \Sc (\EExt^1(A,\OO_X)) \\
  E &\mapsto \Fc(E).
\end{align*}
If we fix an ample divisor $\omega$ on $X$ and $r \in H^0(X), D \in H^2(X)$ such that $r, \omega^2 D$ are coprime,  then from \ref{para:PTstabobjproperties} we know that an object $E$ with $\ch_0(E)=-r$ and $\ch_1(E)=-D$  lies in $\EE_0$ if and only if it is a PT stable object.  Let us write
\[
 \mathcal{H} = \{ A \in \Coh^{=3}(X): \mathrm{hd}(A)\leq 1, A \text{ is $\mu_\omega$-stable}\}.
\]
Then  we can further restrict $\Fc$ to:
\begin{align*}
  \Fc_{PT,r,D} : \{E \text{ is a PT stable object} : \ch_0(E)=-r, \ch_1(E)=-D\} &\to \coprod_{\substack{A \in \mathcal{H}\\ \ch_0(A)=r, \, \ch_1(A)=D}} \Sc (\EExt^1 (A,\OO_X)) \\
  E &\mapsto \Fc (E)
\end{align*}

\begin{thm}\label{thm:main1}
The functors $\Fc, \Fc_{\mathrm{tf}}$ are essentially surjective.  If $\omega$ is an ample divisor on $X$ and $r \in H^0(X), D \in H^2(X)$ are such that $r,\omega^2 D$ are coprime, then  $\Fc_{PT,r,D}$ is also essentially surjective.
\end{thm}

\begin{proof}
The essential surjectivity of $\Fc$ follows from Lemma  \ref{lem:AG44-97-lem1}, while that of $\Fc_{\mathrm{tf}}, \Fc_{PT,r,D}$ follows from the essential surjectivity of $\Fc$ itself and the discussion in \ref{para:Fcrestrictions}.
\end{proof}

\section{Fibers of the functor $\Fc$}\label{sec:Fcfiber}

Given an object $E$ of $\EE_0$, the functor $\Fc$ constructed in Section \ref{sec:esssurj} takes $E$ to a surjective morphism of coherent sheaves
\[
  \Fc (E) : \EExt^1 (H^{-1}(E),\OO_X) \to H^0(E)^\ast.
\]
In this section, we answer the following questions:
\begin{enumerate}
\item Given an object $E \in \EE_0$, how do we enumerate all the objects $\ol{E} \in \EE_0$ such that $\Fc (E)$ and $\Fc (\ol{E})$ are isomorphic in $\Mor (\Coh (X))$?
\item Given two objects $E, \ol{E}$ of $\EE_0$ such that $\Fc (E), \Fc (\ol{E})$ are isomorphic in $\Mor (\Coh (X))$, precisely when  are $E, \ol{E}$ isomorphic in $\EE_0$?
\end{enumerate}
These questions are answered in Lemma \ref{lem:AG44-136-1} and Lemma \ref{lem:AG44-135-1}, respectively.

For the purpose of  computing invariants, however, it may help to think of $\Fc (\ol{E})$ as a point of a quot scheme.
\begin{defn}
For any coherent sheaf $F$ on $X$, we will write $\Qc (F)$ to denote the subcategory of $\Mor (\Coh (X))$ where the set of objects is
\[
  \Qc (F) = \{ F \overset{q}{\to} Q : q \text{ is a surjection in $\Coh (X)$}, Q \in \Coh^{\leq 0}(X)\},
\]
and where the morphisms from an object $F \xrightarrow{q_1} Q_1$ to another $F \xrightarrow{q_2} Q_2$ are commutative diagrams in $\Coh (X)$ of the form
\[
\xymatrix{
  F \ar[r]^{q_1} \ar[d]_{1_F} & Q_1 \ar[d]^f \\
  F \ar[r]^{q_2} & Q_2
}.
\]
\end{defn}
A morphism in $\Qc (F)$ as above is an isomorphism if and only if $f$ is an isomorphism in $\Coh (X)$.  Note that $\Qc (F)$ is not  a full subcategory of $\Mor (\Coh (X))$, i.e.\ $\Qc (F)$ has `fewer' arrows than $\Mor (\Coh (X))$.

If we have an isomorphism $E \cong \ol{E}$ in $\EE_0$,  then  this isomorphism induces an isomorphism   $\Fc (E) \cong \Fc (\ol{E})$  in $\Mor (\Coh (X))$ by virtue of $\Fc$ being a functor.  When $H^{-1}(E)=H^{-1}(\ol{E})$, however, it is not necessarily the case that an isomorphism $E \cong \ol{E}$ in $\EE_0$  induces an isomorphism in the category $\Qc (\EExt^1(H^{-1}(E),\OO_X))$.  If $E, \ol{E} \in \EE_0$ satisfy $H^{-1}(E)=H^{-1}(\ol{E})$ and $\Fc (E), \Fc (\ol{E})$ are isomorphic as objects of $\Qc (\EExt^1(H^{-1}(E),\OO_X))$, however, it is indeed true that $E$ and $\ol{E}$ are isomorphic in $D^b(X)$ by  \ref{cor:AG44-137-1}.  We phrase this formally in Lemma \ref{cor:AG44-137-1plus}.

In the last part of this section, we revisit a construction mentioned in \cite[p.3]{gholampour2017higher} in which a surjection of sheaves $\EExt^1 (I_C,\OO_X) \twoheadrightarrow Q$, where $I_C$ is the ideal sheaf of a Cohen-Macaulay curve $C$ on $X$ and $Q$ is a 0-dimensional sheaf, gives rise to a PT stable pair (i.e.\ a rank-one PT stable object).  We generalise this construction to higher ranks, and compare the generalisation with the functor $\Fc$ constructed in Section \ref{sec:esssurj}.

\begin{lem}\label{lem:AG44-136-1}
Fix an element $\ol{E} \in \EE_0$.  The functor
\begin{align*}
\Gc : \{ E \in \EE_0 :\Fc (E) \cong \Fc (\ol{E})\} &\to
\{ A \in \Coh (X): \mathrm{hd}(A)\leq 1, \EExt^1 (A,\OO_X) \cong \EExt^1 (H^{-1}(\ol{E}),\OO_X)\} \notag \\
E &\mapsto H^{-1}(E) \label{eq:functorGdef}
\end{align*}
is essentially surjective.
\end{lem}

\begin{proof}
That $\Gc$ is a functor from the stated domain to the stated codomain  follows from the definitions of $\EE$ and the construction of the functor $\Fc$.  To see the essential surjectivity of $\Gc$, take any $A$ in the codomain of $\Gc$ and fix an isomorphism of sheaves $\alpha : \EExt^1 (A, \OO_X) \to \EExt^1 (H^{-1}(\ol{E}), \OO_X)$.  Let $c$ denote the canonical map $A^\vee [1] \to H^0(A^\vee [1])$.
 The composite map
\begin{equation}\label{eq:AG44-136-eq1}
  A^\vee [1] \overset{c}{\to} H^0(A^\vee [1]) \overset{\alpha}{\to} H^0(H^{-1} (\ol{E})^\vee [1]) \xrightarrow{\Fc(\ol{E})} H^0( \ol{E})^\vee [3]
\end{equation}
 can  be completed to an exact triangle in $D^b(X)$
\[
 A^\vee [1] \xrightarrow{\Fc(\ol{E})\circ \alpha \circ c} H^0( \ol{E})^\vee [3] \to G^\vee [3] \to A^\vee [2]
\]
for some object $G$ (which is unique up to a non-canonical isomorphism \cite[TR3, 1.2]{huybrechts2006fourier}).  Applying $-^\vee [3]$ gives us the second row of the following diagram; the first row is constructed using truncation functors, while the unmarked vertical maps are canonical and induced by the identity map on $G$ \cite[IV.4 Lemma 5b]{GM}:
\[
\xymatrix{
  H^{-1}(G)[1] \ar[d] \ar[r]  & G \ar[d]^1 \ar[r] & H^0(G) \ar[d] \ar[r] & H^{-1}(G)[2] \ar[d]\\
  A[1] \ar[r] & G \ar[r] & H^0(\ol{E}) \ar[r] & A[2]
}.
\]
Applying $-^\vee [2]$ to the entire diagram above now gives the isomorphism of exact triangles
\[
\xymatrix{
  H^0(\ol{E})^\vee [2] \ar[r] \ar[d] & G^\vee [2] \ar[r] \ar[d]^1 & A^\vee [1] \ar[r]^(.45){\Fc(\ol{E})\circ \alpha \circ c} \ar[d] & H^0 (\ol{E})^\vee [3] \ar[d] \\
  H^0(G)^\vee [2] \ar[r] & G^\vee [2] \ar[r] & H^{-1}(G)^\vee [1] \ar[r] & H^0(G)^\vee [3]
}
\]
in which the right-most square factorises as
\begin{equation}\label{eq:AG44-136-eq3}
\xymatrix{
 A^\vee [1] \ar[r]^(.45)c \ar[d] & H^0 (A^\vee [1]) \ar[rr]^{\Fc (\ol{E}) \circ \alpha} \ar[d] && H^0(\ol{E})^\vee [3] \ar[d] \\
 H^{-1}(G)^\vee [1] \ar[r]^(.45){c'} & H^0 (H^{-1}(G)^\vee [1]) \ar[rr]^(.55){(H^0 \wt{\Fc})(G)} && H^0(G)^\vee [3]
}
\end{equation}
through  canonical maps $c, c'$.

Since $\ol{E} \in \EE_0$ by assumption, the morphism of sheaves $\Fc(\ol{E})$, and hence $\Fc(\ol{E}) \circ \alpha$, is surjective.  It follows that $(H^0\wt{\Fc})(G)$ is also surjective by the commutativity of the right-hand square of \eqref{eq:AG44-136-eq3}, and so  $G \in \EE_0$ by Lemma \ref{lem:surjiffCohleq0vanish}.  Hence we can write $(H^0\wt{\Fc})(G)$ as $\Fc (G)$.

Note that in  \eqref{eq:AG44-136-eq3}, the map $c'$ is canonical and all the vertical maps are isomorphisms.  Taking inverses of the vertical maps in the right-hand square in \eqref{eq:AG44-136-eq3}, the following concatenation gives an isomorphism $\Fc(G) \to \Fc(\ol{E})$:
\[
\xymatrix{
H^0( H^{-1}(\ol{E})^\vee [1]) \ar[r]^(.55){\Fc(\ol{E})} & H^0(\ol{E})^\vee [3] \\
H^0 (A^\vee [1]) \ar[u]^\alpha \ar[r]^{\Fc (\ol{E}) \circ \alpha} & H^0(\ol{E})^\vee [3] \ar@{=}[u] \\
 H^0 (H^{-1}(G)^\vee [1]) \ar[r]^(.55){\Fc(G)} \ar[u] & H^0(G)^\vee [3] \ar[u]
}.
\]
That is, $G$ is an object in the domain of the functor $\Gc$ such that $\Gc (G) = H^{-1}(G) \cong A$ (this isomorphism follows from the left-most vertical map in \eqref{eq:AG44-136-eq3}), proving the essential surjectivity of   $\Gc$.
\end{proof}

\subparagraph Given a fixed object $\ol{E} \in \EE_0$, the proof of Lemma  \ref{lem:AG44-136-1} says  we can construct all the objects in  $\{ E \in \EE_0 :\Fc (E) \cong \Fc (\ol{E})\}$ by first going through all the coherent sheaves $A$ of homological dimension at most 1 for which there exists  a sheaf isomorphism $\alpha : \EExt^1 (A,\OO_X) \to \EExt^1 (H^{-1}(\ol{E}),\OO_X)$, and then completing composite maps \eqref{eq:AG44-136-eq1} to exact triangles.

For objects $E, \ol{E} \in \EE_0$, the following lemma gives a comparison between the condition of $E, \ol{E}$ being isomorphic in $D^b(X)$ and the condition of ${\Fc} (E),  \Fc (\ol{E})$ being isomorphic in $\Mor (\Coh (X))$.

\begin{lem}\label{lem:AG44-135-1}
Given $E, \Eb \in \EE_0$, the following are equivalent:
\begin{itemize}
\item[(i)] $E, \Eb$ are isomorphic in $D^b(X)$.
\item[(ii)] There exists an isomorphism $\Fc (E) \to \Fc (\Eb)$ in $\Mor (\Coh (X))$
    \begin{equation}\label{eq:lemAG44-135-1}
    \xymatrix{
      H^0(H^{-1}(\Eb)^\vee [1]) \ar[r]^(.55){\ol{u}} \ar[d]^j & H^0 (\Eb)^\vee [3] \ar[d]^k \\
      H^0(H^{-1}(E)^\vee [1]) \ar[r]^(.55)u & H^0(E)^\vee [3]
    }
    \end{equation}
    and an isomorphism $H^{-1}(\Eb)^\vee [1] \overset{i}{\to} H^{-1}(E)^\vee [1]$ in $D^b(X)$ such that $H^0(i)=j$.
\end{itemize}
\end{lem}

\begin{proof}
The implication (i) $\Rightarrow$ (ii) follows from the construction of the functor $\Fc$.

Let us now assume  (ii) holds.  We can concatenate \eqref{eq:lemAG44-135-1} with the commutative square induced by $i$ to form
\[
\xymatrix{
H^{-1} (\Eb)^\vee [1] \ar[r]^(.44){\ol{c}} \ar[d]^i &  H^0(H^{-1}(\Eb)^\vee [1]) \ar[r]^(.55){\ol{u}} \ar[d]^j & H^0 (\Eb)^\vee [3] \ar[d]^k \\
      H^{-1}(E)^\vee [1] \ar[r]^(.44)c & H^0(H^{-1}(E)^\vee [1]) \ar[r]^(.55)u & H^0(E)^\vee [3]
}
\]
where $\ol{c}, c$ are canonical maps.  Since $\Fc (E)={u}$ and $\Fc (\Eb)=\ol{u}$ in our notation, from the construction of $\Fc$ we know the composite maps $\bar{u} \bar{c}, uc$ can be completed to exact triangles with $\ol{E}^\vee [3], E^\vee [3]$:
\begin{equation}
\xymatrix{
H^{-1} (\Eb)^\vee [1] \ar[r]^(.5){\bar{u}\bar{c}} \ar[d]^i &   H^0 (\Eb)^\vee [3] \ar[d]^k \ar[r] & \Eb^\vee [3] \ar@{.>}[d]^l  \ar[r] & H^{-1}(\ol{E})^\vee [2] \ar[d]^{i[1]} \\
      H^{-1}(E)^\vee [1] \ar[r]^(.5){uc}    & H^0(E)^\vee [3] \ar[r] & E^\vee [3]  \ar[r] & H^{-1}(E)^\vee [2]
}
\end{equation}
and then $i, k$ induce an isomorphism $l$ in $D^b(X)$, and so (i) holds.
\end{proof}

\subparagraph Suppose \label{cor:AG44-137-1} $E, \Eb \in \EE_0$ satisfy $H^{-1}(E)=H^{-1}(\Eb)$ and $\Fc (E),  \Fc (\Eb)$ are isomorphic in $\Qc (H^{-1}(E))$.  Then by taking $i=1_{{H^{-1}(E)}^\vee [1]}$ in the proof of Lemma \ref{lem:AG44-135-1}, we see that $E, \Eb$ are isomorphic in $D^b(X)$.  Let us phrase this in a slightly more formal framework in Lemma \ref{cor:AG44-137-1plus} below.

\paragraph Let us define a subcategory $\EE_0'$ of $\EE_0$ where the objects of $\EE_0'$ are the same as those of $\EE_0$, but where a morphism $E \overset{f}{\to} \ol{E}$ in $\EE_0'$ is a morphism in $\EE_0$ with the extra requirement that, in the induced morphism of exact triangles
\[
\xymatrix{
  H^{-1}(E)[1] \ar[r] \ar[d]^{H^{-1}(f)[1]} & E \ar[r] \ar[d]^f & H^0(E) \ar[r] \ar[d] & H^{-1}(E)[2] \ar[d] \\
  H^{-1}(\ol{E})[1] \ar[r] & \ol{E} \ar[r] & H^0(\ol{E}) \ar[r] & H^{-1}(\ol{E})[2]
}
\]
we must have
\[
  H^{-1}(f) = 1_{H^{-1}(E)}.
\]
As a result, if $E, \ol{E}$ are two objects of $\EE_0$ such that $H^{-1}(E) \neq H^{-1}(\ol{E})$, then $\Hom_{\EE_0'}(E,\ol{E}) = \varnothing$, meaning  $\EE_0'$ is a non-full subcategory of $\EE_0$.  It is easy to see  that $\Fc$ restricts to a functor
\[
  \Fc' : \EE_0' \to \coprod_{F \in \Coh (X), \, \mathrm{hd}(F)\leq 1} \Qc (\EExt^1(F,\OO_X))
\]
For a fixed object $\ol{E}$ of $\EE_0'$, we also define the full subcategory of $\EE_0'$
\[
\EE_0'(\ol{E}) := \{ E \in\EE_0' : H^{-1}(E)=H^{-1}(\ol{E}), \Fc' (E) \cong \Fc' (\ol{E}) \text{ in } \Qc (\EExt^1 (H^{-1}(\ol{E}),\OO_X)) \}.
\]

\begin{lem}\label{cor:AG44-137-1plus}
Let $\ol{E}$ be a fixed object of $\EE_0'$.  Then $\Fc'$ further restricts to a functor
\begin{equation*}
 \EE_0'(\ol{E}) \to \Qc (\EExt^1 (H^{-1}(\ol{E}),\OO_X))
\end{equation*}
which induces an injection from the set of isomorphism classes in  the domain  to the set of isomorphism classes in the codomain.
\end{lem}

\begin{proof}
Suppose $E, \ol{E}$ are two objects of $\EE_0'$ such that $H^{-1}(E)=H^{-1}(\ol{E})$ and $\Fc' (E) \cong \Fc' (\ol{E})$ in the category $\Qc (\EExt^1 (H^{-1}(\ol{E}),\OO_X))$.  Then we have a commutative diagram in $D^b(X)$
\[
\xymatrix{
  H^{-1}(E)^\vee [1] \ar[r]^(.4)c \ar[d]^1 & \EExt^1 (H^{-1}(E),\OO_X) \ar[r]^(.58)q \ar[d]^1 & H^0(E)^\vee [3] \ar[d] \\
  H^{-1}(\ol{E})^\vee[1] \ar[r]^(.4)c & \EExt^1 (H^{-1}(\ol{E}),\OO_X) \ar[r]^(.58){\ol{q}} & H^0(\ol{E})^\vee [3]
}
\]
where $c$ is the canonical map, the right-hand vertical arrow is an isomorphism, and $q=\Fc'(E), \ol{q} = \Fc' (\ol{E})$.  From the construction of the functor $\Fc$, we can complete $qc, \ol{q}c$ to exact triangles with $E^\vee [2], \ol{E}^\vee [2]$ and obtain an isomorphism of exact triangles
\[
\xymatrix{
E^\vee [2] \ar[r] \ar@{.>}[d]^{g^\vee [2]} &  H^{-1}(E)^\vee [1] \ar[r]^{qc} \ar[d]^1 & H^0(E)^\vee [3] \ar[d] \ar[r] & E^\vee [3] \ar@{.>}[d]^{g^\vee [3]} \\
\ol{E}^\vee [2] \ar[r] & H^{-1}(\ol{E})^\vee [1] \ar[r]^{\ol{q}c} & H^0(\ol{E})^\vee [3] \ar[r] & E^\vee [3]
}
\]
for some isomorphism $\ol{E} \overset{g}{\to} E$ in $D^b(X)$.  Dualising, we obtain the isomorphism of triangles
\[
\xymatrix{
   H^{-1}(\ol{E})[1] \ar[d]^1 \ar[r] & \ol{E} \ar[d]^g \ar[r] & H^0(\ol{E}) \ar[d] \ar[r] & H^{-1}(\ol{E})[2] \ar[d]^1 \\
   H^{-1}(E)[1] \ar[r] & E \ar[r] & H^0(E) \ar[r] & H^{-1}(E)[2]
}
\]
from which we see $H^{-1}(g)=1_{H^{-1}(\ol{E})}$ by \cite[IV.4 Lemma 5b)]{GM}.  That is, the morphism $g$ is a morphism in the category $\EE_0'$.
\end{proof}

\begin{lem}\label{lem:E0objrankone}
We have
\[
  \{ E \in \EE_0 : \ch_0(E)=-1, \ch_1(E)=0, H^{-1}(E) \text{ is torsion-free}\} = \{ \text{rank-one PT stable objects}\}.
\]
\end{lem}

\begin{proof}
The inclusion from left to right follows from \ref{para:PTstabobjproperties}.  To see the other inclusion, take any rank-one PT stable object $E$.  Then $E \in \Ap$, and the canonical exact triangle \eqref{eq:Fdef-eq1} makes   $H^{-1}(E)[1]$ an $\Ap$-subobject of $E$.  Since $\Coh^{\leq 0}(X)$ is closed under quotient in $\Ap$  and we have the vanishing $\Hom (\Coh^{\leq 0}(X),E)=0$ from the PT stability of $E$, it follows that $\Hom (\Coh^{\leq 0}(X),H^{-1}(E)[1])=0$.   Lemma \ref{lem:vanishingCoh01} then implies $\mathrm{hd}(H^{-1}(E))\leq 1$, and so $E$ lies in the category on the left-hand side.
\end{proof}


\paragraph For \label{para:GKrank1constr} a fixed Cohen-Macaulay curve $C$ on a smooth projective threefold $X$, Gholampour-Kool describes a construction that takes an element of $\Qc (\EExt^1(I_C,\OO_X))$ to a stable pair on $X$ \cite[p.3]{gholampour2017higher}.  This construction   is as follows: given a surjection $q : \EExt^1 (I_C,\OO_X) \to Q$ in $\Coh (X)$ where $Q$ is 0-dimensional, let $K = \kernel (q)$ so that we have a short exact sequence of sheaves
\[
0 \to K \to \EExt^1 (I_C,\OO_X) \overset{q}{\to} Q \to 0.
\]
Taking derived dual and then taking cohomology, and noting that $\EExt^1 (I_C,\OO_X) \cong \OO_C^\ast$ where $\OO_C$ is reflexive, we obtain the short exact sequence
\[
0 \to \OO_C \to K^\ast \to Q^\ast \to 0.
\]
Taking the Yoneda product of the last exact sequence with the structural exact sequence
\begin{equation}\label{eq:structuresesC}
0 \to I_C \to \OO_X \to \OO_C \to 0
\end{equation}
then gives the four-term exact sequence
\[
0 \to I_C \to \OO_X \overset{s}{\to} K^\ast \to Q^\ast \to 0
\]
and hence a stable pair
\begin{equation}\label{eq:AG44-153-4}
\OO_X \overset{s}{\to} F
\end{equation}
 where $F := K^\ast$.

\paragraph The  \label{para:GKrank1constr-higherrk} construction in    \ref{para:GKrank1constr} can be generalised to higher ranks.  This is already mentioned in \cite[Lemma 3.3]{Toda2}, but we describe the details here so that we can compare the construction with our functor $\Fc$:  Suppose $A = H^{-1}(E)$ for some $E \in \EE_0$, and that $A$ is torsion-free with  $A^{\ast \ast}$  locally free.   For any surjection $q : \EExt^1 (A,\OO_X) \to Q$ in $\Coh (X)$ where $Q$ is 0-dimensional, let $K = \kernel (q)$.  We have a natural short exact sequence of sheaves
\begin{equation}\label{eq:AG44-153-1}
0\to A \overset{\beta}{\to} A^{\ast \ast}\overset{\gamma}{\to} T \to 0
\end{equation}
where $T \in \Coh^{\leq 1}(X)$.  Taking derived dual and noting that $A^{\ast \ast}$ is locally free, we obtain the isomorphism
\[
  \EExt^1 (A,\OO_X) \to \EExt^2 (T,\OO_X)
\]
where $\EExt^2 (T,\OO_X)$ is a pure sheaf in $\Coh^{\leq 1}(X)$ by \cite[Proposition 1.1.6]{HL}.  Dualising brings the short exact sequence
\[
0 \to K \to \EExt^1 (A,\OO_X) \overset{q}{\to} Q \to 0
\]
to the short  exact sequence
\[
0 \to \EExt^2(\EExt^1 (A,\OO_X),\OO_X) \overset{\delta}{\to} \EExt^2 (K,\OO_X) \overset{\epsilon}{\to} Q^\ast \to 0
\]
where we have an isomorphism $\EExt^2(\EExt^2 (T,\OO_X),\OO_X) \overset{i}{\to} \EExt^2(\EExt^1 (A,\OO_X),\OO_X)$.  Note that if $T$ is nonzero, then it must be pure 1-dimensional.  To see this, notice that the short exact sequence of sheaves \eqref{eq:AG44-153-1} induces the short exact sequence in $\Ac^p$
\[
0 \to T \to A[1] \to A^{\ast \ast}[1] \to 0.
\]
Since $\Coh^{\leq 0}(X)$ is a Serre subcategory of  $\Ac^p$,  any nonzero 0-dimensional subsheaf $T'$ of $T$ would be an $\Ac^p$-subobject of $A[1]$, hence of $E$, contradicting the vanishing $\Hom (\Coh^{\leq 0}(X),E)=0$ from the definition of $\EE_0$.  Hence when $T$ is nonzero, it is pure 1-dimensional, hence reflexive \cite[Proposition 1.1.10]{HL}, and so the double dual map  $T \overset{\alpha}{\to} \EExt^2(\EExt^2(T,\OO_X),\OO_X)=T^{\ast \ast}$ is an isomorphism.  Putting everything together, the composition of morphisms in $\Coh (X)$
\[
  A^{\ast \ast} \overset{\gamma}{\twoheadrightarrow} T \overset{\alpha}{\to} T^{\ast \ast} =\EExt^2(\EExt^2(T,\OO_X),\OO_X) \overset{i}{\to}  \EExt^2 (\EExt^1 (A,\OO_X),\OO_X) \overset{\delta}{\hookrightarrow} \EExt^2(K,\OO_X)
\]
 has $A$ as the kernel and a 0-dimensional sheaf as the cokernel.  Overall, we have constructed a 2-term complex
  \begin{equation}\label{eq:AG44-153-3}
   A^{\ast \ast} = H^{-1}(E)^{\ast \ast} \xrightarrow{s} \EExt^2 (K,\OO_X)
 \end{equation}
 where $s = \delta i\alpha\gamma$.   When $A=I_C$ is the ideal sheaf of a Cohen-Macaulay curve $C$ on $X$, we have $T=\OO_C$, and the above construction reduces to the construction in \ref{para:GKrank1constr} while the 2-term complex \eqref{eq:AG44-153-3}  coincides with the stable pair \eqref{eq:AG44-153-4}.

Recall that the proof of Lemma \ref{lem:AG44-97-lem1} describes a construction  taking any element of $\Qc (\EExt^1(I_C,\OO_X))$ to a rank-one PT stable object.    The following lemma describes the precise relation between this construction in the proof of Lemma \ref{lem:AG44-97-lem1} and the construction in \ref{para:GKrank1constr-higherrk}.

\begin{lem}\label{lem:inversecompGK}
Let $E$ be an object of $\EE_0$ such that $H^{-1}(E)$ is torsion-free and $H^{-1}(E)^{\ast \ast}$ is locally free.  Suppose $q : \EExt^1 (H^{-1}(E),\OO_X) \to Q$ is a surjection in $\Coh (X)$ where $Q$ is 0-dimensional.  Let $G$ be the object in $\EE_0$ satisfying $\Fc(G) \cong q$ as constructed in the proof of Lemma \ref{lem:AG44-97-lem1}, and let $[H^{-1}(E)^{\ast\ast} \overset{s}{\to} \EExt^2(K,\OO_X)]$ be the 2-term complex \eqref{eq:AG44-153-3} constructed from $q$ as in \ref{para:GKrank1constr-higherrk}.   Then $G$ fits in an exact triangle
\[
  (H^{-1}(E)^\ast)^\vee \overset{\phi}{\to} K^\vee [2] \to G \to (H^{-1}(E)^\ast)^\vee [1],
\]
and $s$ and $G$ are related by $H^0(\phi)=s$.
\end{lem}

\begin{proof}
Let us use the notation in \ref{para:GKrank1constr-higherrk} and  write $A = H^{-1}(E)$, $K=\kernel (q)$.  Let us also write $c$ to denote the canonical map $A^\vee [1] \to \EExt^1(A,\OO_X)$ as in the proof  of Lemma \ref{lem:AG44-97-lem1}.  Let $C$ be an object that completes $qc$ to an exact triangle
\[
  A^\vee [1] \overset{qc}{\to} Q \to C \to A^\vee [2].
\]
 Applying the octahedral axiom to the composition $qc$ then gives us the diagram
\[
\xymatrix{
& & & A^\ast [2] \ar[ddd] \\
& & & \\
& \EExt^1 (A,\OO_X) \ar[dr]^q \ar[uurr]^a & & \\
A^\vee [1] \ar[rr]^{qc} \ar[ur]^c & & Q \ar[r] \ar[dr] & C  \ar[d] \\
& & & K[1]
}
\]
in which every straight line is an exact triangle.  Taking $-^\vee [3]$ then yields
\[
\xymatrix{
& & & (A^\ast)^\vee [1]  \ar[ddll]_{a^\vee [3]} \\
& & & \\
& (\EExt^1(A,\OO_X))^\vee [3]  \ar[dl]_{c^\vee [3]}  & & \\
A [2]   & & Q^\vee [3] \ar[ll]_{(qc)^\vee [3]} \ar[ul]_{q^\vee [3]} & G   \ar[uuu] \ar[l] \\
& & & K^\vee [2] \ar[ul] \ar[u]
}
\]
where we write $G = C^\vee [3]$.  The vertical line in the last diagram gives an exact triangle
\[
  (A^\ast)^\vee  \overset{\phi}{\to} K^\vee [2] \to G \to (A^\ast)^\vee [1].
\]
Now we apply the formulation of the octahedral axiom in \cite[Lemma 1.4.6]{neeman2014triangulated} to the composition $K^\vee [2] \to G \to Q^\vee [3]$, which yields the diagram
\[
\xymatrix{
K^\vee [2] \ar[r]  \ar[d]^1 & G \ar[d] \ar[r] & (A^\ast)^\vee [1] \ar[d]^{a^\vee [3]} \ar[r]^{\phi [1]} & K^\vee [3] \ar[d]^1 \\
K^\vee [2] \ar[r] \ar[d] & Q^\vee [3] \ar[r] \ar[d] & (\EExt^1(A,\OO_X))^\vee [3] \ar[r] \ar[d] & K^\vee [3] \ar[d] \\
0 \ar[r] \ar[d] & A[2] \ar[r]^1 \ar[d] & A[2] \ar[r] \ar[d] & 0 \ar[d] \\
K^\vee [3] \ar[r] & G[1] \ar[r] & (A^\ast)^\vee [2] \ar[r]^{\phi [2]} & K^\vee [4]
}
\]
in which every row and every column is an exact triangle.  Applying the cohomology functor $H^{-1}$ to the top-right square gives
\[
\xymatrix{
  A^{\ast \ast} \ar[r]^{H^0(\phi)} \ar[d]^{H^0(a^\vee [2])} & \EExt^2 (K,\OO_X) \ar[d]^1 \\
  \EExt^2 (\EExt^1(A,\OO_X),\OO_X) \ar[r]^(.6)\delta & \EExt^2 (K,\OO_X)
}
\]
where $\delta$ is as in \ref{para:GKrank1constr-higherrk}, and $H^0(a^\vee [2])$ is precisely the surjection $i\alpha \gamma$ in \ref{para:GKrank1constr-higherrk}.  Thus $H^0(\phi)$ coincides with $s$, which was constructed as $\delta i \alpha \gamma$.
\end{proof}

\begin{eg}\label{eg:lempairrepresentation}
 Suppose $E$ is an object satisfying the hypotheses of Lemma \ref{lem:inversecompGK}, such that  $H^{-1}(E)$ is not locally free while its dual $H^{-1}(E)^{\ast}$ is locally free.  Using the notation in \ref{para:GKrank1constr-higherrk} and Lemma  \ref{lem:inversecompGK}, we have that $T$ is nonzero and hence pure 1-dimensional by \ref{para:GKrank1constr-higherrk}.  It follows that $\EExt^1 (A,\OO_X)$ and $K$ are both nonzero and pure 1-dimensional, as is $\EExt^2 (K,\OO_X)=K^\ast$.  On the other hand, we have  $(H^{-1}(E)^\ast)^\vee = H^{-1}(E)^{\ast \ast}$.  Hence $\phi = H^0(\phi)$ and the morphism $\phi$ coincides with the morphism $s$ in Lemma  \ref{lem:inversecompGK}.  Exampes of such $E$ include:
\begin{enumerate}
\item $E$ is a rank-one PT stable object.  In this case, we have $H^{-1}(E)=L \otimes I_C$ where $L$ is  some line bundle  and $I_C$ is the ideal sheaf of some Cohen-Macaulay curve $C$ on $X$, so that   $\phi = s : L \to \EExt^2 (K,\OO_X)$.  When $L=\OO_X$, this morphism is a PT stable pair.
\item $E$ is a 2-term complex of coherent sheaves $[G \overset{\varphi}{\to} F]$ in $D^b(X)$ with $F$ sitting at degree 0, and where $(G,F,\varphi)$ is a stable frozen triple (see Example \ref{eg:stabfrozentripleinEE0}).  To see why $H^{-1}(E)$ satisfies the requirements of Lemma \ref{lem:inversecompGK}, consider the short exact sequence of sheaves
    \[
      0 \to H^{-1}(E) \to G \to \image (\varphi) \to 0
    \]
    where $\image \varphi$ is a pure 1-dimensional sheaf.  This gives the isomorphism $\EExt^1(H^{-1}(E),\OO_X) \cong \EExt^2 (\image (\varphi),\OO_X)\neq 0$, and so $H^{-1}(E)$ cannot be locally free.  We also have $H^{-1}(E)^\ast = \EExt^0 (H^{-1}(E),\OO_X) \cong \EExt^0 (G,\OO_X)$, which is locally free.  Hence $E$ satisfies all the requirements in Lemma  \ref{lem:inversecompGK}.
\end{enumerate}
\end{eg}

\bibliography{refs}{}
\bibliographystyle{plain}

\end{document}